\documentclass[12pt,reqno]{amsart}
\pdfoutput=1 
\usepackage{amsmath,bbm}
\usepackage{latexsym}
\usepackage{amsfonts}
\usepackage{amssymb}
\usepackage{color}
\usepackage[table]{xcolor}

\usepackage{graphicx}
\usepackage{url}
\usepackage{enumerate}
\usepackage{tikz}
\usetikzlibrary{shadings, intersections, calc, plotmarks}
\usepackage{marginnote}
\usepackage{stackrel}
\usepackage{empheq}
\usepackage{tcolorbox}
\usetikzlibrary{arrows.meta,positioning}

\usepackage{geometry}
\usepackage{hyperref}
\hypersetup{
  colorlinks=true,
  linkcolor=blue!90!black,    
  citecolor=green!50!black,   
  urlcolor=blue!70!black,     
  linktoc=all,                
}
\usepackage{cleveref}

\xdefinecolor{tumblue}     {RGB}{0,101,189}
\xdefinecolor{tumgreen}    {RGB}{162,173,  0}
\xdefinecolor{tumred}      {RGB}{229, 52, 24}
\xdefinecolor{tumivory}    {RGB}{218,215,203}
\xdefinecolor{tumorange}   {RGB}{227,114, 34}
\xdefinecolor{tumlightblue}{RGB}{152,198,234}

\newtheorem*{proposition*}{Proposition}

\newtheorem{theorem}{Theorem}
\newtheorem*{theorem*}{Theorem}
\newtheorem{lemma}{Lemma}
\newtheorem{corollary}{Corollary}
\newtheorem*{corollary*}{Corollary}

\newtheorem{example}{Example}

\newcommand{\Sym}{\operatorname{Sym}}

\newcommand{\R}{\mathbbm{R}}

\newcommand{\C}{\mathbbm{C}}

\newcommand{\cL}{\mathcal{L}}

\newcommand{\cU}{\mathcal{U}}

\newcommand{\1}{\mathbbm{1}}

\def\>{{\rangle}}
\def\<{{\langle}}

\def\per{{\rm per}\,}
\newcommand{\be}{\begin{equation}}
\newcommand{\ee}{\end{equation}}
\newcommand{\bea}{\begin{eqnarray}}
\newcommand{\eea}{\end{eqnarray}}

\newcommand{\tr}[1]{\mathrm{tr}\left[#1\right]} 

\newcommand{\rank}[1]{\mathrm{rank}\left[#1\right]}

\begin{document}

\title{Poset-refined majorization relations}

\author[Guterman]{Alexander E. Guterman}
\email{alexander.guterman@biu.ac.il}
\address{Bar-Ilan University, Ramat-Gan, Israel}

\author[Wolf]{Michael M. Wolf$^{1,2}$}
\email{m.wolf@tum.de}
\address{$^1$ Department of Mathematics, Technical University of Munich}
\address{$^2$ Munich Center for Quantum
Science and Technology (MCQST),  Munich, Germany}

\date{\today}

\begin{abstract} Several classical majorization relations for sums or products of matrices involve a majorizing vector of perfectly aligned eigenvalues or singular values. By relaxing the order of alignment to a partial order, we show that the majorization can be strengthened, provided the change-of-basis matrices admit an LU-approximation
with respect to this partial order. In this way, we obtain refined versions of Ky Fan's majorization relations, Horn's log-majorization relation, and von Neumann's trace inequality. As an application, we give a short proof of the separable Ky Fan majorization relation for an arbitrary number of tensor factors and extend it to a sum of tensor products of arbitrary matrices. Further applications concern majorization relations for sums of (anti-)symmetric powers and for products of Kronecker sums.
\end{abstract}
\maketitle
\section{Introduction}\label{sec:intro}
We will study three closely related majorization relations: Ky Fan's majorization relation for eigenvalues of Hermitian matrices, Ky Fan's weak majorization relation for singular values, and Horn's $\log$-majorization relation for singular values \cite[Ch.9]{MOmajo}:
\begin{equation}\label{eq:intro}
    \begin{aligned}
    \lambda(A+B)&\prec \lambda(A)+\lambda(B),\\
    \sigma(A+B)&\prec_w\sigma(A)+\sigma(B),\\
\sigma(AB)&\prec_{\log}\sigma(A)*\sigma(B).
\end{aligned}
\end{equation}
On the right-hand side of each relation, the entries of the two vectors are combined in aligned order, both vectors being arranged decreasingly. Our aim is to sharpen these bounds by changing this alignment, replacing the linear order on the index set with a partial order. Such a refinement cannot hold unconditionally, but it will be shown to hold whenever the change-of-basis matrix between $A$ and $B$ admits an LU-approximation
with respect to the partial order under consideration.

This work was inspired by \cite{Alhejji2026}, where a majorization relation is established for the sum of two tensor products of positive matrices. The proof given there relies on an intricate linear-programming argument, which raises the question of whether a more general framework underlies the result and could yield it (and others like it) with less effort. We develop such a framework. In \Cref{cor:tensors} we use it to extend the result of \cite{Alhejji2026} to arbitrary (not necessarily positive) matrices and to any number of tensor factors.

\section{Preliminaries}\label{sec:prelim}

We begin by fixing some notation.
Let $P$ be a finite set of cardinality $|P|=d$. By ${P\choose k}$ we denote the set of $k$-element subsets of $P$.  A matrix $A\in\C^{P\times P}$ can be regarded as a map $P\times P\rightarrow\C, (x,y)\mapsto A_{x,y}$. That is, its rows and columns are indexed in $P$. The usual way of labeling the entries of a matrix is then recovered if we choose $P$ to be $[d]\coloneqq\{1,\ldots,d\}$.  For subsets $S,T\subseteq P$, $A_{S,T}$ will denote the $|S|\times|T|$ submatrix of $A$ whose rows are specified by $S$ and columns by $T$.

We will write $\sigma_1(A)\geq\cdots\geq\sigma_d(A)$ for the decreasingly ordered list of singular values, and if $A$ is Hermitian, $\lambda_1(A)\geq\cdots\geq\lambda_d(A)$ for the eigenvalues. For any  $k\in [d]$, we will write $\sigma_{[k]}(A)\coloneqq\sum_{i=1}^k\sigma_i(A)$ (the \emph{Ky Fan $k$-norm} of $A$) and, for a Hermitian matrix, $\lambda_{[k]}(A)$ will denote the sum of the $k$ largest eigenvalues. Similarly, for any function $f:P\rightarrow\R$ and subset $S\subseteq P$, we will use the shorthand $f(S):=\sum_{x\in S}f(x)$.  By $D_f\in\C^{P\times P}$ we will denote the diagonal matrix with diagonal entries $\big(f(x)\big)_{x\in P}$.
Moreover, we write $\1$ for the identity matrix, whose dimension is determined by the context.

Two vectors $x,y\in\R^d$ satisfy the \emph{weak majorization} relation $x\prec_w y$ if $\sum_{i=1}^k x_i^\downarrow\leq \sum_{i=1}^k y_i^\downarrow$ holds for all $k\in[d]$. Here, the down arrows denote a decreasingly ordered rearrangement of the entries. \emph{Majorization} $x\prec y$ requires that in addition equality holds for $k=d$.
For vectors with nonnegative entries, we write $x\prec_{\log}y$ for the $\log$-majorization relation, which demands that
$$
    \prod_{i=1}^k x_i^\downarrow\leq \prod_{i=1}^k y_i^\downarrow\quad\text{for all }k\in[d],
$$
with equality for $k=d$. Finally, we use $x*y$ for componentwise multiplication.

\subsection*{Eigenvalue bounds} We will need two lemmas that bound eigenvalues of a sum of two Hermitian matrices in terms of sums of the individual eigenvalues:
\begin{lemma}\label{lem:ApB}
    Let $(a_i)_{i=1}^d$ and $(b_j)_{j=1}^d$ be orthonormal eigenbases of Hermitian matrices $A,B\in\C^{d\times d}$ with $Aa_i=\lambda_i(A)a_i$ and $Bb_j=\lambda_j(B)b_j$. Then
    \begin{equation}\label{eq:lambdamax}
        \lambda_{\max}(A+B)\leq\max_{\langle a_i,b_j\rangle\neq 0}\big(\lambda_i(A)+\lambda_j(B)\big)\eqqcolon M
    \end{equation}
\end{lemma}
\begin{proof}
    This follows from
    \begin{equation*}
\lambda_{\max}(A+B)=\lim_{t\rightarrow\infty}\frac1t \log\tr{e^{t(A+B)}}\leq\lim_{t\rightarrow\infty}\frac1t\log\tr{e^{tA}e^{tB}}=M.
    \end{equation*}
    Here, the middle step is the Golden-Thompson inequality, and the equalities on the left and right both use the elementary limit 
    $$\lim_{t\rightarrow\infty}\frac1t\log\sum_k c_k e^{t m_k}=\max\{m_k\mid c_k>0\}\quad\text{for}\quad c\in\R_{\geq 0}^d\setminus\{0\}.$$ The two equalities follow by evaluating the traces in the eigenbasis of $(A+B)$, $A$, and $B$, respectively.
\end{proof}

For a Hermitian matrix $A\in\C^{d\times d}$ and $k\in[d]$, we define an endomorphism $A^{[k]}$ on the $k$'th exterior power of $\C^d$ via
\begin{equation}
    A^{[k]}(v_1\wedge\cdots\wedge v_k)\coloneqq\sum_{i=1}^k v_1\wedge\cdots\wedge A v_i\wedge\cdots\wedge v_k.
\end{equation}
As a matrix,  this is called the $k$'th \emph{additive compound matrix} of $A$, and it satisfies $(A+B)^{[k]}=A^{[k]}+B^{[k]}$ and $\lambda_{\max}(A^{[k]})=\lambda_{[k]}(A)$ \cite{compound}.
\begin{lemma}\label{lem:detVST}
Let $P$ be a set of cardinality $d$, and let  $$A=\sum_{i\in P}\alpha(i)|a_i\rangle\langle a_i|,\qquad B=\sum_{j\in P}\beta(j)|b_j\rangle\langle b_j|$$ be spectral decompositions of two Hermitian matrices in orthonormal bases.
Set $C_{ij}\coloneqq\langle a_i,b_j\rangle$. Then for all $k\in[d]$:
    \begin{equation}
        \lambda_{[k]}(A+B)\leq \max_{S,T\in\binom{P}{k}}\big\{\alpha(S)+\beta(T)\mid \det C_{S,T}\neq 0\big\}.\label{eq:lambdakdet}
    \end{equation}
\end{lemma}
\begin{proof}
    For $\lambda_{[k]}(A+B)=\lambda_{\max}(A^{[k]}+B^{[k]})$ we can apply \Cref{lem:ApB} and use that the vectors $\bigwedge_{i\in S}a_i$, with $S\in\binom{P}{k}$, form an eigenbasis of $A^{[k]}$ with corresponding eigenvalues $\alpha(S)$ (and similarly for $B^{[k]}$). The overlap between two eigenvectors is therefore $\big|\langle \bigwedge_{i\in S}a_i,\bigwedge_{j\in T}b_j\rangle\big|=\big|\det C_{S,T}\big|$, see \cite[Ex. I.5.1]{Bhatia}, so that Eq.(\ref{eq:lambdakdet}) follows from Eq.(\ref{eq:lambdamax}).
\end{proof}
\subsection*{Partially ordered sets }
A \emph{partially ordered set} (poset) is a set $P$ with a reflexive, antisymmetric, transitive relation for which we write ``$x\leq y$''.  A function $f:P\rightarrow\R$ is called \emph{order-decreasing} if $x\leq y\Rightarrow f(x)\geq f(y)$, where the usual order on $\R$ is used. A poset is called  \emph{linearly ordered} (or \emph{totally ordered}) if every two elements can be compared, i.e., if $x\leq y \vee y\leq x$.
For a Cartesian product $P_1\times\cdots\times P_n$ of posets, the \emph{product order} is defined componentwise by $$(x_1,\ldots,x_n)\leq(y_1,\ldots,y_n)\quad\Longleftrightarrow\quad \forall i: x_i\leq y_i.$$

\section{Refined majorization from poset-LU factorization}\label{sec:KFLU}

For a finite poset $P$, we define the sets of lower and upper triangular matrices 
\begin{equation}
    \begin{aligned}
 \cL(P)&:=\{L\in\C^{P\times P}\mid L_{xy}=0\text{ unless }y\leq x\},\\
 \cU(P)&:=\{U\in\C^{P\times P}\mid U_{xy}=0\text{ unless }x\leq y\}.
\end{aligned}                               \label{eq:poset-triangular}
\end{equation}


If for $C\in\C^{P\times P}$ there are sequences $L_n\in\cL(P)$, and $U_n\in\cU(P)$ s.t. $\lim_{n\rightarrow\infty}L_n U_n=C$ in the norm topology, then we say that $C$ \emph{admits an LU-approximation}.  If $P=[d]$ with the usual order, then all matrices admit an LU-approximation.
For general posets, however, not every matrix admits an LU-approximation\footnote{Consider for instance a poset where no two distinct elements are comparable. Then only the diagonal matrices are triangular.}. 
In fact, we are going to prove that \emph{every} matrix admits an LU-approximation if and only if the poset is a linearly ordered set:

\begin{lemma} \label{lem:linear_order}
    Let $P$ be a poset, $|P|=d$. Then the closure in the norm topology
    $$\overline{\cL(P)\cU(P)}=\C^{P\times P}$$
    if and only if $P$ is linearly ordered.
\end{lemma}
\begin{proof}
If $P$ is linearly ordered, we can identify $P$ with $[d]$ equipped with the usual order. Then an invertible matrix $M$ admits an LU-factorization $M=LU$ with $L\in\cL(P), U\in\cU(P)$ iff all leading principal minors are nonzero \cite[Cor. 3.5.4]{horn13}. Since the excluded set is a proper algebraic variety, this implies density in the norm topology, i.e., $$\overline{\cL(P)\cU(P)}=\C^{P\times P}.$$ 
    
To  prove the `only if' part, let $\dim_\C\cU(P)=d+k$, i.e., $k$ is the number of  comparable distinct pairs in $P$. Define $\cL_1(P)\coloneqq\{L\in\cL(P)|L_{xx}=1 \text{ for all } x\in P\}$, and the polynomial map $$\Phi:\cL_1(P)\times\cU(P)\rightarrow\C^{P\times P},\qquad (L,U)\mapsto LU.$$ 
Then $\dim_\C\cL_1(P)=k$, so that the domain of $\Phi$ has algebraic dimension $d+2k$ and the Zariski closure of its image has dimension at most $d+2k$.

Every invertible $M\in\cL(P)\cU(P)$ belongs to the image of $\Phi$: indeed, if $M=LU$ is invertible, then $L$ and $D\coloneqq{\mathrm diag}(L)$ are invertible, and we have $$M=(LD^{-1})(DU),$$
where $LD^{-1}\in\cL_1(P)$ and $DU\in\cU(P)$. If $P$ is not linearly ordered, then $k<\binom{d}{2}$, and therefore
$$d+2k<d+2\binom{d}{2}=d^2.$$
Thus the image of $\Phi$ is contained in a proper algebraic subset of $\C^{P\times P}$, and since the invertible matrices are dense, this means that $\cL(P)\cU(P)$ cannot be dense unless $k=\binom{d}{2}$.

\end{proof}

If the change-of-basis matrix of a pair of Hermitian matrices does admit an LU-approximation, then stronger majorization relations can be obtained:
\begin{theorem}[Refined Ky Fan majorization for eigenvalues]\label{thm:KFlambda} Let $P$ be a finite poset, $\alpha,\beta: P\rightarrow\R $ order-decreasing, and $$A=\sum_{i\in P}\alpha(i)|a_i\rangle\langle a_i|,\qquad B=\sum_{j\in P}\beta(j)|b_j\rangle\langle b_j|$$ spectral decompositions in orthonormal bases indexed by $P$. If $C_{ij}:=\langle a_i,b_j\rangle$ satisfies $C\in\overline{\cL(P)\cU(P)}$, then
    \begin{equation}\label{eq:KyFanP}
        \lambda(A+B)\prec \lambda(D_{\alpha+\beta})\prec\lambda(A)+\lambda(B),
    \end{equation}
    where $D_{\alpha+\beta}$ is the diagonal matrix with diagonal entries $\big(\alpha(x)+\beta(x)\big)_{x\in P}$.
\end{theorem}
\emph{Remark:} Note that for $P=[d]$ with the standard order, this reduces to the classical Ky Fan relation since, in this case, $\lambda(D_{\alpha+\beta})=\lambda(A)+\lambda(B)$.
\begin{proof}
    Set $d:=|P|$. The left majorization relation in Eq.(\ref{eq:KyFanP}) is a consequence of the fact that the following chain of inequalities holds for any $k\in[d]$, with  the equality at $k=d$:
    \begin{equation}\label{eq:KFPproofineqs}
         \lambda_{[k]}(A+B)\leq\max_{S,T\in\binom{P}{k}}\{\alpha(S)+\beta(T)\mid\det C_{S,T}\neq 0\}\leq \lambda_{[k]}(D_{\alpha+\beta}).
    \end{equation}
    The first inequality in Eq.(\ref{eq:KFPproofineqs}) is just \Cref{lem:detVST}. To prove the second inequality, consider any $S,T\in\binom{P}{k}$ with $\det C_{S,T}\neq 0$. Choose sequences $L_n\in\cL(P)$, $U_n\in\cU(P)$ with $L_n U_n\rightarrow C$. By continuity of the determinant,  $\det (L_n U_n)_{S,T}\neq 0$ holds for sufficiently large $n$. We fix one such $n$ and apply Cauchy-Binet:
    \begin{equation}
        \det(L_n U_n)_{S,T} =\sum_{R\in\binom{P}{k}} \det\big((L_n)_{S,R}\big)\det\big((U_n)_{R,T}\big)\neq 0.
    \end{equation}
    Hence, there is some $R$ for which $\det\big((L_n)_{S,R}\big)\neq 0$ and $\det\big((U_n)_{R,T}\big)\neq 0$. We will show first that $\alpha(S)\leq\alpha(R)$. Adding a constant to $\alpha$ changes both $\alpha(S)$ and $\alpha(R)$ by the same amount, since $|S|=|R|$. We can therefore assume that $\alpha\geq 0$. For $t\in[0,\infty)$, define
    $$F_t\coloneqq\{x\in P\mid \alpha(x)> t\},\qquad S_t\coloneqq S\cap F_t,\qquad R_t\coloneqq R\cap F_t .$$
    Suppose that $s\in S_t$ and $ r\in R\setminus R_t$. Then $\alpha(r)<\alpha(s)$, which implies $s\ngeq r$, so that   $ (L_n)_{s,r}=0$ since $L_n$ is lower triangular. Consequently, all nonzero columns of $(L_n)_{S_t,R}$  already occur in $(L_n)_{S_t,R_t}$. Since $(L_n)_{S,R}$ is nonsingular, all its rows, especially those indexed by $S_t$, are linearly independent. Therefore 
    $$|S_t|=\rank{(L_n)_{S_t,R}}=\rank{(L_n)_{S_t,R_t}}\leq |R_t|.$$
    By nonnegativity of $\alpha$, we can for any $x\in P$ write $\alpha(x)=\int_0^\infty \chi(t) dt$, with $\chi(t)\coloneqq 1$ if $\alpha(x)>t$ and $\chi(t)=0$ otherwise. Summing this identity term by term and using the preceding cardinality inequality then gives
    \begin{equation}
        \alpha(S)=\int_0^\infty |S_t|\; dt\leq \int_0^\infty |R_t|\; dt=\alpha(R).
    \end{equation}
    Similarly, by considering columns instead of rows and upper-triangularity instead of lower-triangularity, we show that $\beta(T)\leq\beta(R)$. Thus
    $$\alpha(S)+\beta(T)\leq \alpha(R)+\beta(R)\leq\lambda_{[k]}\big(D_{\alpha+\beta}\big),$$ which proves the second inequality in Eq.(\ref{eq:KFPproofineqs}).

    Finally, the right majorization relation in Eq.(\ref{eq:KyFanP}) is simply a consequence of the fact that the sum of any $k$ entries of a real vector never exceeds the sum of its $k$ largest entries. Hence, the aligned sum majorizes any non-aligned sum.
\end{proof}
Before we continue, it might be illustrative to look at a simple example:
\begin{example}\label{ex:ex1}
    Let $P=\{p,q,r\}$ with $p\leq r$ and $q\leq r$, while $p$ and $q$ are incomparable. In the order $(p,q,r)$, take\vspace*{-6pt}
    $$ \alpha=(3,4,0),\qquad \beta=(4,3,0),\qquad
 C=\begin{bmatrix}0&0&1\\0&1&0\\1&0&0\end{bmatrix}.\vspace*{-5pt}$$
 Then $C$ admits an LU-approximation $\lim_{n\rightarrow\infty} L_nU_n=C$ with
 $$
 L_n=\left[
\begin{array}{c c c}
  \cellcolor{gray!25}1 & 0
      & 0  \\
  0 & \cellcolor{gray!25}1
      & 0  \\ \cellcolor{gray!25}
  n & \cellcolor{gray!25}0
      & \cellcolor{gray!25}1
\end{array}
\right],\quad
 \quad
 U_n =
\left[
\begin{array}{c c >{\columncolor{gray!25}}c}
  \cellcolor{gray!25}1/n & 0
      & 1  \\
  0 & \cellcolor{gray!25}1
      & 0  \\
  0 & 0
      & -n
\end{array}
\right],\quad L_n U_n
 =\begin{bmatrix}1/n&0&1\\0&1&0\\1&0&0\end{bmatrix}.
$$
Here, the gray boxes indicate all entries that are allowed to be nonzero under the conditions $L_n\in\cL(P)$ and $U_n\in\cU(P)$.
\Cref{thm:KFlambda} applied to $A=D_\alpha, B=CD_\beta C^*$ s.t. $A+B={\rm diag}(3,7,4)$ then gives $\lambda(A+B)\prec(7,7,0)$, whereas the classical Ky Fan majorization relation only yields the weaker statement $\lambda(A+B)\prec(8,6,0)$.
\end{example}
Applying convex trace functions, any majorization relation yields a multitude of additional inequalities. Here, we just state one of them: 
\begin{corollary}[Refined von Neumann trace inequality]\label{cor:vNtrace}
    Under the assumptions of \Cref{thm:KFlambda}, for Hermitian matrices $A,B\in\C^{P\times P}$ with order-decreasing spectral functions $\alpha,\beta:P\rightarrow\R$, and a change-of-basis matrix $C\in\overline{\cL(P)\cU(P)}$, we have
    \begin{equation}\label{eq:vNtrace}
        \tr{AB}\leq\sum_{x\in P}\alpha(x)\beta(x).
    \end{equation}
    More generally, if $f,g:I\rightarrow\R$ are  nondecreasing on an interval $I\subseteq\R$ containing the spectra of $A,B$, then 
    \begin{equation}\label{eq:vNtracefg}
        \tr{f(A)g(B)}\leq\sum_{x\in P}f\big(\alpha(x)\big)g\big(\beta(x)\big).
    \end{equation}
\end{corollary}
\begin{proof}
    Due to convexity of the function $t\mapsto t^2$, the majorization relation in \Cref{thm:KFlambda} implies
    \begin{equation}
        \begin{aligned}
            2\tr{AB} &=\tr{(A+B)^2}-\tr{A^2+B^2}\\
            &\leq\sum_{x\in P}\big(\alpha(x)+\beta(x)\big)^2-\sum_{x\in P}\big(\alpha(x)^2+\beta(x)^2\big)\\
            &=2\sum_{x\in P}\alpha(x)\beta(x).
        \end{aligned}
    \end{equation}
    Eq.(\ref{eq:vNtracefg}) follows from Eq.(\ref{eq:vNtrace}) since $f\circ\alpha$, and $g\circ\beta$ are still order-decreasing. For instance, $x\leq y\Rightarrow\alpha(x)\geq\alpha(y)\Rightarrow f(\alpha(x))\geq f(\alpha(y))$.
\end{proof}
\begin{example}[Information-theoretic quantities] Let $A=\rho_\alpha, B=\rho_\beta$ be  density matrices and $\alpha,\beta$ the corresponding spectral functions. Then Eq.(\ref{eq:vNtracefg}) yields bounds on various information theoretic quantities by specifying certain $f,g$.

If we choose $f(t)=t, g(t)=\log(t)$, then Eq.(\ref{eq:vNtracefg}) leads to the relative entropy bound
\begin{equation}
    D(\rho_\alpha\|\rho_\beta)\geq\sum_{x\in P}\alpha(x)\log\frac{\alpha(x)}{\beta(x)}.
\end{equation}
If we choose $f(t)=t^s, g(t)=t^{1-s}$ for $s\in(0,1)$ we obtain a bound on the \emph{quantum Chernoff coefficient} via
\begin{equation}
    \tr{\rho_\alpha^s\rho_\beta^{1-s}}\leq\sum_{x\in P}\alpha(x)^s \beta(x)^{1-s}.
\end{equation}
\end{example}

Now we drop the Hermiticity assumption, and look at the singular values of a sum of matrices:
\begin{theorem}[Refined Ky Fan majorization for singular values]\label{thm:KF} Let $P$ be a finite poset, $\alpha,\beta: P\rightarrow[0,\infty) $ order-decreasing, and $$A=VD_\alpha W,\qquad B=\tilde{V}D_\beta\tilde{W}$$ singular value decompositions of matrices in $\C^{P\times P}$.  If  $V^*\tilde V,W\tilde W^*\in\overline{\cL(P)\cU(P)}$, then\vspace*{-5pt}
    \begin{equation}\label{eq:KyFanPsigma}
        \sigma(A+B)\prec _w\lambda(D_{\alpha+\beta})\prec\sigma(A)+\sigma(B),
    \end{equation}
    where $D_{\alpha+\beta}$ is the diagonal matrix with diagonal entries $\big(\alpha(x)+\beta(x)\big)_{x\in P}$. We note that for a non-negative diagonal matrix its singular values coincide with its eigenvalues. 
\end{theorem}
\begin{proof}
    We can write $A+B=LR$ with
$$
 L\coloneqq \begin{pmatrix}VD_\alpha^{1/2}&\tilde V D_\beta^{1/2}\end{pmatrix},
 \qquad
R\coloneqq\begin{pmatrix}D_\alpha^{1/2}W\\D_\beta^{1/2}\tilde W\end{pmatrix}.
$$ Applying \Cref{thm:KFlambda}, we obtain
\begin{equation*}
    \begin{aligned}
    \lambda_{[k]}(LL^*)&=\lambda_{[k]}\big(VD_\alpha V^*+\tilde VD_\beta \tilde V^*\big)\leq \lambda_{[k]}(D_{\alpha+\beta}),\ \text{ and}\\
    \lambda_{[k]}(R^* R)&=\lambda_{[k]}\big(W^*D_\alpha W+\tilde W^* D_\beta \tilde W\big)\leq \lambda_{[k]}(D_{\alpha+\beta}).
    \end{aligned}
\end{equation*}
Using Ky Fan's variational formula and the Hilbert--Schmidt Cauchy--Schwarz inequality, this leads to the first majorization relation in Eq.(\ref{eq:KyFanPsigma}) since
\begin{equation*}
    \sigma_{{[k]}}(A+B)=\max_{E^*E=F^*F=\1_k}\big|\tr{E^*LRF}\big|\leq\sqrt{\lambda_{[k]}(LL^*)\lambda_{[k]}(R^*R)}\leq \lambda_{[k]}(D_{\alpha+\beta}).
\end{equation*} The second majorization relation again follows from the fact that the aligned sum majorizes every non-aligned sum.
\end{proof}

Finally, we turn to the singular values of a product of matrices:

\begin{theorem}[Refined Horn log-majorization for singular values]\label{thm:HornP}
Let $P$ be a finite poset, let $\alpha,\beta:P\rightarrow[0,\infty)$ be order-decreasing, and let
$$
    A=VD_\alpha W,\qquad B=\tilde V D_\beta\tilde W
$$
be singular value decompositions of matrices in $\C^{P\times P}$. If
$
    C:=W\tilde V\in\overline{\cL(P)\cU(P)},
$
then\vspace*{-5pt}
\begin{equation}\label{eq:HornP}    \sigma(AB)\prec_{\log}\lambda(D_{\alpha\beta})\prec_{\log}\sigma(A)*\sigma(B),
\end{equation}
where $D_{\alpha\beta}$ is the diagonal matrix with diagonal entries $\big(\alpha(x)\beta(x)\big)_{x\in P}$.
\end{theorem}
\begin{proof}
Set $d:=|P|$ and fix $k\in[d]$. Define $\hat{\alpha},\hat{\beta}:\binom{P}{k}\rightarrow[0,\infty)$ via
$$
    \hat{\alpha}(S):=\prod_{x\in S}\alpha(x),\qquad
    \hat{\beta}(S):=\prod_{x\in S}\beta(x).
$$
Since $AB=VD_\alpha C D_\beta\tilde W$, its singular values equal those of $D_\alpha C D_\beta$. Let $\hat{C}:=\wedge^k C$ be the $k$'th exterior power of $C$. Its matrix entries are the minors $(\det C_{S,T})_{S,T\in\binom{P}{k}}$ and since $C$ is unitary, so is $\hat{C}$. From the properties of exterior powers \cite[Ch.I.5]{Bhatia} we obtain
\begin{equation}
     \prod_{i=1}^k\sigma_i(AB)
    =\prod_{i=1}^k \sigma_i(D_\alpha C D_\beta)=\big\|\wedge^k(D_\alpha C D_\beta)\big\|_\infty=\big\|D_{\hat{\alpha}}\hat{C}D_{\hat{\beta}}\big\|_\infty.\label{eq:sigmaAB}
\end{equation}
   Next, we apply Cordes' inequality \cite{Cordes}, which states that positive semidefinite operators $X,Y$ satisfy $\|X^s Y^s\|_\infty\leq\|XY\|_\infty^s$ for all $s\in(0,1]$, to $X\coloneqq D_{\hat{\alpha}}^{1/s}$, $Y\coloneqq \hat{C}D_{\hat{\beta}}^{1/s}\hat{C}^*$. 
In this way, we can bound Eq.(\ref{eq:sigmaAB}) via
\begin{align*}
\big\|D_{\hat{\alpha}}\hat{C}D_{\hat{\beta}}\big\|_\infty&\leq \big\|D_{\hat{\alpha}}^{1/s}\hat{C}D_{\hat{\beta}}^{1/s}\big\|_\infty^s\leq \tr{D_{\hat{\alpha}}^{2/s}\hat{C}D_{\hat{\beta}}^{2/s}\hat{C}^*}^{s/2}\\
&=\bigg[\sum_{S,T\in\binom{P}{k}}\big(\hat{\alpha}(S)\hat{\beta}(T)\big)^{2/s}\big|\hat{C}_{S,T}\big|^2\bigg]^{s/2}. 
\end{align*}
Here, the second inequality is the Hilbert-Schmidt bound $\|\cdot\|_\infty\leq\|\cdot\|_2$.

Combining with Eq.(\ref{eq:sigmaAB}) and taking the limit $s\rightarrow 0$, we obtain 
\begin{equation}\label{eq:Hornsupport}
    \prod_{i=1}^k\sigma_i(AB)
    \leq\max_{\substack{S,T\in\binom{P}{k}\\ \det C_{S,T}\neq0}}
    \hat{\alpha}(S)\hat{\beta}(T).
\end{equation}
For each pair $S,T$ on the right, the Cauchy--Binet argument in the proof of \Cref{thm:KFlambda} gives $R\in\binom{P}{k}$ and nonzero minors of a lower and an upper triangular matrix $L_n$ on $(S,R)$ and $U_n$ on $(R,T)$, respectively. Since $\det (L_n)_{S,R}\neq 0$, the Leibniz expansion contains a nonzero term. Hence, there exists a bijection $\pi:S\rightarrow R$ such that $(L_n)_{s,\pi(s)}\neq 0$ for every $s\in S$. As $L_n\in\cL(P)$, this implies $\pi(s)\leq s$ for all $s\in S$. Similarly, $\det(U_n)_{R,T}\neq 0$ yields a bijection $\tau:R\rightarrow T$ satisfying $r\leq \tau(r)$ for all $r\in R$. Therefore
$$
    \hat{\alpha}(S)\leq\hat{\alpha}(R),\qquad
    \hat{\beta}(T)\leq\hat{\beta}(R).
$$
Together with Eq.(\ref{eq:Hornsupport}), this proves
$$
    \prod_{i=1}^k\sigma_i(AB)
    \leq\max_{R\in\binom{P}{k}}\hat{\alpha}(R)\hat{\beta}(R)
    =\prod_{i=1}^k\lambda_i(D_{\alpha\beta}).
$$
Finally,
$$
    \prod_{i=1}^k\lambda_i(D_{\alpha\beta})
    \leq\max_{\substack{S,T\in\binom{P}{k}}}
    \hat{\alpha}(S)\hat{\beta}(T)
    =\prod_{i=1}^k\sigma_i(A)\sigma_i(B).
$$
For $k=d$, all three products equal $|\det(AB)|$, which proves Eq.(\ref{eq:HornP}).
\end{proof}
\addtocounter{example}{-2}
\begin{example}[continued] From \Cref{thm:HornP}
    we obtain the relation $(12,0,0)\prec_{\log} (12,12,0)\prec_{\log} (16,9,0)$. So again, the intermediate vector is sharper than the classical  bound.
\end{example}


\section{Applications with product structure}\label{sec:apps}

In this section, we discuss various applications that have in common that there is some form of product structure.  This then gives rise to a natural choice of the partial order -- the product order.

\begin{lemma}\label{lem:tensor}
    Let $P=[d_1]\times\cdots\times[d_n]$ be equipped with the product order. For arbitrary $C_i\in\C^{d_i\times d_i}$, we have
    $$C_1\otimes\cdots\otimes C_n\in\overline{\cL(P)\cU(P)}.$$
\end{lemma}
\begin{proof}
    Choose lower and upper triangular matrices $L_k^{(i)}\in\cL([d_i]),U_k^{(i)}\in\cU([d_i])$ such that $L_k^{(i)}U_k^{(i)}\rightarrow C_i$ for $k\rightarrow\infty$. Then
    $$\bigotimes_i L_k^{(i)}\in\cL(P),\qquad \bigotimes_i U_k^{(i)}\in\cU(P),$$
    since a nonzero matrix entry forces the corresponding inequality in every coordinate. Moreover,
    $$\Big(\bigotimes_i L_k^{(i)}\Big)\Big(\bigotimes_i U_k^{(i)}\Big)=\bigotimes_i\big(L_k^{(i)}U_k^{(i)}\big)\longrightarrow\bigotimes_i C_i .$$
\end{proof}

Now we are able to prove our motivating example. For positive matrices a proof of this relation has recently appeared in \cite{alhejji2026majorizationrelationsumtensor}.

\begin{corollary}[Separable Ky Fan majorization]\label{cor:tensors}
Let $A_i,B_i\in \C^{d_i\times d_i}$ be arbitrary.  Then\vspace*{-5pt}
\begin{equation}\label{eq:SepFan}
    \sigma\!\left(\bigotimes_{i=1}^nA_i+
            \bigotimes_{i=1}^nB_i\right)
 \prec_w
 \left(\bigotimes_{i=1}^n \sigma(A_i)+
       \bigotimes_{i=1}^n \sigma(B_i)\right)^\downarrow.  
\end{equation}

If all $A_i$ and $B_i$ are positive semidefinite, then singular values become eigenvalues, and weak majorization becomes majorization.
\end{corollary}

\begin{proof}
We consider $P=[d_1]\times\cdots\times[d_n]$  equipped with the product order. Then \vspace*{-5pt}
$$
 \alpha(x)=\prod_{i=1}^n \sigma_{x_i}(A_i),
 \qquad
 \beta(x)=\prod_{i=1}^n \sigma_{x_i}(B_i),
$$
 are order-decreasing functions from  $P$ to $[0,\infty)$ that label the singular values of $A\coloneqq\otimes_i A_i$ and $B\coloneqq\otimes_i B_i$, respectively.  Moreover, the relative left and right singular-basis
matrices are tensor products of local unitaries. Hence,  \Cref{lem:tensor} allows us to apply \Cref{thm:KF}, which yields $$\sigma\big(A+B\big)\prec_w\lambda(D_{\alpha+\beta}),$$ and thus Eq.(\ref{eq:SepFan}). If all matrices are positive semidefinite, then singular values are equal to eigenvalues, and since the sum over all of them is just the trace, we obtain equality for the full sum in Eq.(\ref{eq:SepFan}) -- so we have majorization.
\end{proof}

	In order to obtain the next corollary we need the following combinatorial property:
	\begin{lemma}\label{lem:sort}
		Let $S=(s_1\le\cdots\le s_r)$ and $T=(t_1\le\cdots\le t_r)$ be nondecreasing
		$r$-tuples in $[d]$, and suppose some $\pi\in\mathfrak S_r$ satisfies
		$t_{\pi(a)}\le s_a$ for all $a\in[r]$. Then $t_a\le s_a$ for all $a$, i.e.\ $T\le S$
		in the product order.
	\end{lemma}
	
	\begin{proof}
		Fix $a$ and assume $t_a>s_a$. Since $T$ is nondecreasing, $t_b\ge t_a>s_a$ for all
		$b\ge a$, so at most $a-1$ positions $b$ have $t_b\le s_a$. But for each $c\le a$,
		$t_{\pi(c)}\le s_c\le s_a$ since $S$ is nondecreasing. Also, $\pi(1),\dots,\pi(a)$ are
		$a$ distinct positions. Hence at least $a$ positions $b$ have $t_b\le s_a$. This is  a
		contradiction. Thus $t_a\le s_a$.
	\end{proof}

	\begin{corollary}[Exterior and symmetric powers]\label{cor:powers}
		Let $A,B\in \C^{d\times d}$.
		Then
	\begin{equation}
		\sigma\!\left(\bigwedge\nolimits^{\,\!r} A+\bigwedge\nolimits^{\,\!r} B\right)
		\prec_w
		\left(
		\prod_{t=1}^r \sigma_{i_t}(A)+\prod_{t=1}^r \sigma_{i_t}(B)
		\right)_{1\leq i_1<\cdots<i_r\leq d}^{\downarrow}, \text{ for }r\in[d],  \label{eq:corwedge}
	\end{equation}
		\begin{equation}
		\sigma\big(\Sym^r A+\Sym^r B\big)
		\prec_w
		\left(
		\prod_{t=1}^r \sigma_{i_t}(A)+\prod_{t=1}^r \sigma_{i_t}(B)
		\right)_{1\leq i_1\leq\cdots\leq i_r\leq d}^{\downarrow},  \text{ for }r\geq 1. \label{eq:symcor}
		\end{equation}
		Here $\bigwedge\nolimits^{\,\!r}$ and $\Sym^r$ are the  exterior and symmetric
		powers.  If $A,B\geq0$, then $\sigma$ can be replaced by $\lambda$ and
		$\prec_w$ by $\prec$ in both statements.
	\end{corollary}
	\begin{proof}
		Both powers are continuous, multiplicative, and preserve unitarity \cite{Greub}. Multiplicativity means that for arbitrary $X,Y\in\C^{d\times d}$, we have $$\wedge^r(XY)=(\wedge^r X)(\wedge^r Y),\quad\text{ and }\quad \Sym^r(XY)=\Sym^r(X)\Sym^r(Y).$$
		
		For $\bigwedge^r$, we index the wedge basis $(e_i)_{i\in P}$ by strictly increasing tuples, i.e., we choose the base set $P_{ext}\coloneqq\{i\in[d]^r\mid i_1<\cdots< i_r\}$ 
		equipped with the product order of $[d]^r$. Then $\alpha(i)\coloneqq\prod_{t=1}^r\sigma_{i_t}(A)$ and $\beta(i)\coloneqq\prod_{t=1}^r\sigma_{i_t}(B)$ are order-decreasing maps $P_{ext}\rightarrow[0,\infty)$ that label the singular values of $\wedge^r A$ and $\wedge^r B$, respectively. Due to multiplicativity, the change-of-basis matrices have the form $\wedge^r C$ for some $C\in\C^{d\times d}$. In order to apply \Cref{thm:KF}    it remains to show that $\wedge^r C\in\overline{\cL(P_{ext})\cU(P_{ext})}$, i.e., that the exterior power preserves triangularity. 
To do this we notice that  the entries of
		$\wedge^rC$ are the $r\times r$ minors $(\wedge^rC)_{S,T}=\det C_{S,T}$, cf. Cauchy--Binet formula  \cite[\S0.8]{horn13},\cite[Ch.~V]{Greub}. Let $C\in{\mathcal L}([d])$, i.e.\
		$C_{x,y}=0$ unless $y\le x$. By the definition
		$\det C_{S,T}=\sum_{\pi}\operatorname{sgn}(\pi)\prod_{a}C_{s_a,t_{\pi(a)}}$. If  $\det C_{S,T}\ne 0$ it follows that there exists  $\pi$ such that $C_{s_a,t_{\pi(a)}}\ne0$ for all $a$, whence
		$t_{\pi(a)}\le s_a$ and, by Lemma~\ref{lem:sort}, $T\le S$. Thus
		$(\wedge^rC)_{S,T}=0$ unless $T\le S$, i.e.\ 
\begin{equation}\label{eq:TRL}
C\in{\mathcal L}([d]) \Longrightarrow	 \wedge^rC\in \cL(P_{ext}).
\end{equation} 
		
		Replacing ``$y\le x$'' by ``$x\le y$'' throughout proves the  analogue of \eqref{eq:TRL} for upper triangular matrices, namely,   
\begin{equation}\label{eq:TRU}
	C\in{\mathcal U}([d]) \Longrightarrow	 \wedge^rC\in \cU(P_{ext}).
\end{equation} 
		
Since by the Cauchy--Binet formulas the entries of $\wedge^rC$ are polynomials in the entries of $C$, they are continuous. Given an arbitrary $C\in \C^{d\times d}$,  we can choose $L_n\in\cL([d])$,
		$U_n\in\cU([d])$ with $L_nU_n\to C$, since 
		$\overline{\cL([d])\cU([d])} =\C^{d\times d}$ by Lemma~\ref{lem:linear_order}. Then $\wedge^rL_n\in\cL(P_{ext})$,
		$\wedge^rU_n\in\cU(P_{ext})$ by the formulas \eqref{eq:TRL} and \eqref{eq:TRU}. Therefore,
		$$(\wedge^rL_n)(\wedge^rU_n)=\wedge^r(L_nU_n)\to\wedge^rC, \mbox{ so }
		\wedge^rC\in\overline{\cL(P_{ext})\cU(P_{ext})}.$$
		
		Similarly, for $\Sym^r$, we choose the set $P_{sym}\coloneqq\{i\in[d]^r\mid i_1\leq\cdots\leq i_r\}$ 
		equipped with the product order of $[d]^r$. The elements of $P_{sym}$ then label a basis of normalized symmetric tensors. Again, the same choice for $\alpha$ and $\beta$ is order-decreasing and labels the singular values of $\Sym^r A$ and $\Sym^r B$. The entries of $\Sym^{r}C$ are permanents up to positive normalization,
		\[
		(\Sym^{r}C)_{S,T}=\frac{\per C_{S,T}}{\sqrt{\mu(S)\,\mu(T)}},
		\]
		where $C_{S,T}=(C_{s_a,t_b})_{a,b=1}^{r}$, see \cite[p. 271]{Bhatia84},
		\cite{MarcusMinc}. As $\per C_{S,T}=\sum_{\pi}\prod_{a}C_{s_a,t_{\pi(a)}}$, similar to the determinant case, the
		identical argument with Lemma~\ref{lem:sort} gives $(\Sym^{r}C)_{S,T}=0$ unless
		$T\le S$; hence $\Sym^{r}C\in\cL(P_{sym})$.  Since the permanent is a polynomial, this similarly shows that  $\Sym^{r}C\in\overline{\cL(P_{sym})\cU(P_{sym})}$ and that the symmetric power preserves triangularity. That concludes the proof for the singular values and general matrices.

        If $A$ and $B$ are positive semidefinite, then
 $\wedge^rA,\wedge^rB\ge0$, and the singular values coincide
 with eigenvalues and $\sigma$ may be replaced by $\lambda$. Moreover, both sides of
 \eqref{eq:corwedge}  have equal total sums: the left hand side sums to
 ${\rm tr\,} (\wedge^rA+\wedge^rB) ={\rm tr\,} (\wedge^rA)+{\rm tr\,} (\wedge^rB)=\sum_{i\in P_{ext}}\alpha(i)+\sum_{i\in P_{ext}}\beta(i)$,
 which is exactly the sum of the entries of the right hand side. A weak majorization
 between two vectors with equal total sums is a majorization, so $\prec_w$ improves to
 $\prec$. The same holds for~\eqref{eq:symcor}. 
	\end{proof}

\bigskip

For $A_i\in\C^{d_i\times d_i}$ and $a_i\in\R^{d_i}$ we define the \emph{Kronecker sums}
$$\begin{aligned}
    \mathop{\boxplus}\limits_{i=1}^n A_i&\coloneqq \sum_{i=1}^n \1\otimes\cdots\otimes A_i\otimes\cdots\otimes\1\in\bigotimes_{i=1}^n\C^{d_i\times d_i},\\
    \quad \mathop{\boxplus}\limits_{i=1}^n a_i&\coloneqq \sum_{i=1}^n {\bf 1}\otimes\cdots\otimes a_i\otimes\cdots\otimes{\bf 1}\in\bigotimes_{i=1}^{n}\R^{d_i},
\end{aligned} 
$$ where ${\bf 1}\coloneqq(1,\ldots, 1)$.

\begin{corollary}[Products of Kronecker sums]\label{cor:KS}
    Let $A_i,B_i\in\C^{d_i\times d_i}$ be positive semidefinite, $A:=\boxplus_{i=1}^n A_i$, $B:=\boxplus_{i=1}^n B_i$, $a:=\boxplus_{i=1}^n\lambda(A_i)$, $b:=\boxplus_{i=1}^n\lambda(B_i)$. Then\vspace*{-4pt}
    \begin{equation}
        \sigma(AB)\prec_{\log} (a*b)^\downarrow\prec_{\log}a^\downarrow *b^\downarrow=\lambda(A)*\lambda(B).
    \end{equation}
\end{corollary}
\begin{proof}
    We use the product order on $P=[d_1]\times\cdots\times[d_n]$. Since product eigenvectors diagonalize $A$ and $B$, we can label their eigenvalues by the order-decreasing functions $\alpha,\beta:P\rightarrow[0,\infty)$:
    $$\alpha(x)\coloneqq\sum_{i=1}^n \lambda_{x_i}(A_i)=a_{x},\qquad \beta(x)\coloneqq\sum_{i=1}^n \lambda_{x_i}(B_i)=b_x.$$
    Moreover, the change-of-basis matrix is a tensor product $C=\bigotimes_{i=1}^n C_i$ so that \Cref{lem:tensor} guarantees $C\in\overline{\cL(P)\cU(P)}$ and the result follows from \Cref{thm:HornP}.
\end{proof}

\section{Discussion}
We have seen that the product order gives rise to various applications of the presented method in which the obtained majorization relations, as far as we see,  do not easily follow from the classical relations. It would be interesting to know whether there are other classes of examples with a natural choice of the partial order and non-trivial implications. 

For that purpose, it might also be useful to apply the method under weaker assumptions. In fact, if we look into the proofs of \Cref{thm:KFlambda}, \Cref{thm:KF}, and \Cref{thm:HornP}, we see that LU-approximability was only a convenient way of stating sufficiently strong assumptions. What we have really used, and what the proofs show is implied by LU-approximability, is the following condition on the change-of-basis matrices: whenever $\det C_{S,T}\neq 0$, then there is an $R\in\binom{P}{k}$ s.t. $R\preceq_k S\; \wedge\; R\preceq_k T$. Here,  we write $R\preceq_k T$  for two sets $R,T\in\binom{P}{k}$ iff there is a bijection $\tau:R\rightarrow T$ s.t. $r\leq \tau(r)$ for all $r\in R$.

\section*{Acknowledgments}
Support by the German-Israeli Foundation for Scientific Research and Development (GIF) under the grant no. I-3014-304.6/2026 is greatly acknowledged.

%
%

\appendix

\bibliographystyle{halpha}
\bibliography{KyFan}{}\vspace*{15pt}

@Article{Alhejji2026,
author={Alhejji, Mohammad A.},
title={Refining {Ky Fan's} Majorization Relation with Linear Programming},
journal={Annales Henri Poincar{\'e}},
year={2026},
volume={27},
pages={909-932}
}

@misc{alhejji2026majorizationrelationsumtensor,
      title={A majorization relation for a sum of two tensor products of positive semidefinite operators}, 
      author={Mohammad A. Alhejji and Cole Kelson-Packer},
      year={2026},
      eprint={arXiv:2607.07913}
}

@article{compound,
author = {Fiedler, Miroslav},
journal = {Czechoslovak Mathematical Journal},
language = {eng},
number = {3},
pages = {392-402},
publisher = {Institute of Mathematics, Academy of Sciences of the Czech Republic},
title = {Additive compound matrices and an inequality for eigenvalues of symmetric stochastic matrices},
url = {http://eudml.org/doc/12802},
volume = {24},
year = {1974},
}

@book{horn13,
  address = {Cambridge; New York},
  edition = {2nd},
  author = {Horn, Roger A. and Johnson, Charles R.},
  publisher = {Cambridge University Press},
  title = {Matrix Analysis},
  year = 2013
}

@book{Bhatia,
    author = {R. Bhatia},
    title = {Matrix Analysis},
    publisher = {Springer},
    year = 1997
}

@article{Bhatia84,
    author = {R. Bhatia},
    title = {Variation of symmetric tensor powers and permanents},
    volume = {62}, pages = {269-276},
    journal = {Linear Algebra and its Applications},
    year = 1984,
    doi={10.1016/0024-3795(84)90102-2.}
}

@book{MarcusMinc,
    author = {M. Marcus and H. Minc},
    title = {A Survey of Matrix Theory and Matrix Inequalities},
    publisher = {Allyn and Bacon, Boston},
    year = 1964
}

@book{Greub,
    author = {Greub, Werner},
    title = {Multilinear Algebra},
    publisher = {Springer},
    edition   = {2},
    year = 1978
}

@book{Cordes,
author={H.O. Cordes}, title={Spectral Theory of Linear Differential Operators and Comparison Algebras}, publisher={Cambridge University Press}, year={1987}}

@book{MOmajo,
    author = { Marshall, Albert W. and  Olkin, Ingram  and  Arnold, Barry C.},
    title = {Inequalities: Theory of Majorization and Its Applications},
    publisher = {Springer},
    year = 2011
}

\end{document}